\documentclass[11pt]{amsart}
\usepackage[english]{babel}
\usepackage[T1]{fontenc}
\usepackage[latin1]{inputenc}
\usepackage{graphicx}
\usepackage{amsmath,amssymb,amsthm,mathrsfs,txfonts,ifthen,xr}
\usepackage{soul}
\usepackage{array, multirow}
\usepackage{stmaryrd}
\usepackage{color}
\usepackage{hyperref}
\usepackage{enumerate}
\usepackage[all]{xy}
\usepackage{tikz-cd}
\usetikzlibrary{decorations.pathreplacing}

\usepackage{ytableau}

\theoremstyle{plain}
\newtheorem{theo}{Theorem}[section]
\newtheorem*{theo*}{Theorem}
\newtheorem{lemma}[theo]{Lemma}
\newtheorem{proposition}[theo]{Proposition}
\newtheorem{corollary}[theo]{Corollary}

\theoremstyle{definition}
\newtheorem{definition}[theo]{Definition}

\newtheorem{notation}[theo]{Notation} 
\newtheorem{conjecture}[theo]{Conjecture}
\newtheorem*{conjecture*}{Conjecture}
\newtheorem{example}[theo]{Example}

\theoremstyle{remark}
\newtheorem{remark}[theo]{Remark}

\def\A{{\rm A}}
\def\B{{\rm B}}
\def\C{{\rm C}}
\def\D{{\rm D}}
\def\d{{\rm d}}

\def\F{{\rm F}}

\def\H{{\rm H}}
\def\I{{\rm I}}
\def\J{{\rm J}}

\def\L{{\rm L}}
\def\M{{\rm M}}

\def\P{{\rm P}}

\def\Q{{\rm Q}}

\def\R{{\rm R}}
\def\S{{\rm S}}

\def\T{{\rm T}}
\def\U{{\rm U}}
\def\V{{\rm V}}
\def\W{{\rm W}}
\def\X{{\rm X}}
\def\Y{{\rm Y}}

\def\tr{{\rm tr}}

\def\Id{{\rm Id}}
\def\Re{{\rm Re}}

\def\Mat{{\rm Mat}}

\def\log{{\rm log}}

\def\Spec{{\rm Spec}}
\def\diag{{\rm diag}}

\allowdisplaybreaks
\title{Towards a Characterization of  Matrix Monotonicity via Radial Convexity}

\author{Ma\l{}gorzata M. Czerwi\'nska}
\address{Department of Mathematics and Statistics\\ University of North Florida \\ 1 UNF Drive \\ Jacksonville \\ FL 32224 \\ USA}
\email{m.czerwinska@unf.edu}

\author{Jose A. Franco}
\address{Department of Mathematics and Statistics\\ University of North Florida \\ 1 UNF Drive \\ Jacksonville \\ FL 32224 \\ USA}
\email{jose.franco@unf.edu}

\author{Allan Merino}
\address{Department of Mathematics and Statistics\\ University of North Florida \\ 1 UNF Drive \\ Jacksonville \\ FL 32224 \\ USA}
\email{allan.merino@unf.edu}



\date{}

\begin{document}

\begin{abstract}

The theory of matrix monotone functions of finite degree depends subtly on the matrix size. Motivated by this phenomenon, we introduce radial convexity of metrics with respect to matrix orders on the cone of positive definite matrices. We establish radial convexity of several natural metrics with respect to the Loewner order and then focus on the Frobenius norm. Using the Daleckii-Krein formula, we relate radial convexity to the positivity of the Hadamard square of the corresponding Loewner matrix. This leads us to conjecture that radial convexity of order $n$ is equivalent to matrix monotonicity of degree $n$, and we prove this equivalence in degree $2$. We further provide evidence in higher dimensions, obtain a characterization in degree $3$, and discuss extensions of our results to arbitrary unitarily invariant norms.

\end{abstract}

\maketitle

\section{Introduction}

In matrix analysis and operator theory there are two hallmark concepts that undergird the study of several interesting objects. These two concepts are the idea of operator or matrix monotonicity and the study of the geometry of the cone of positive operators or positive definite matrices via different metrics. It should not be surprising that these two concepts have been intertwined in the study of matrix analysis. 

\noindent The study of means of positive definite matrices is closely connected with monotonicity. A consequence of the operator monotonicity of the associated function to a Kubo-Ando mean is the natural order-theoretic in-betweenness property
\begin{equation*}
\A \leq \A \sigma \B \leq \B\,.
\end{equation*}
This property plays an important role in the theory of Kubo-Ando means and provides a direct matrix analogue of scalar betweenness.

\noindent The geometry connection comes with the geometric in-betweenness  proposed by Audenaert in the study of matrix means that do not necessarily satisfy the axioms of Kubo-Ando theory \cite{Audenaert2013}. In this approach, the order relation is replaced by a metric condition, requiring the mean to lie metrically between its arguments. Given a metric $\d$ on the cone $\mathbb{P}_{n}$ of positive definite $n\times n$ matrices,, geometric in-betweenness is expressed by the conditions
\begin{equation*}
\d(\A, \A \sigma \B) \leq \d(\A, \B) \qquad \text{ and } \qquad \d(\B, \A \sigma \B) \leq \d(\A, \B)\,.
\end{equation*}
At first sight, one might expect order-theoretic in-betweenness to imply geometric in-betweenness. This is indeed the case for several familiar geometries on $\mathbb{P}_{n}$. Consider the case of an $f$-metric induced by a matrix monotone function $f$ of degree $n$ as follows,
\begin{equation*}
\d_{\|\cdot\|,f}(\A,\B) :=
\left\|f(\A)-f(\B)\right\|
\end{equation*}
for a unitarily invariant norm $\|\cdot\|$.  Then 
\begin{equation*}
\A \leq \C \leq \B \qquad \Longrightarrow \qquad f(\A) \leq f(\C) \leq f(\B)\,\quad\text{or}\quad f(\A) \geq f(\C) \geq f(\B) ,
\end{equation*}
which implies that
\begin{equation*}
\d_{\left\|\cdot\right\|, f}(\A,\C) \leq \d_{\left\|\cdot\right\|,f}(\A, \B)\,\quad\text{and}\quad\d_{\left\|\cdot\right\|, f}(\B,\C) \leq \d_{\left\|\cdot\right\|,f}(\A, \B).
\end{equation*}
See Proposition \ref{Prop:MonoConv} for more details. This class includes many familiar metrics from matrix analysis, such as those induced by unitarily invariant norms, the Hellinger metric, and the chaotic Thompson metric. 

\noindent However, the example \(f(x)=x^2\) shows that an \(f\)-metric need not be compatible with the Loewner order: even if
\begin{equation*}
\A \leq \C \leq \B\,,
\end{equation*}
it may happen that
\begin{equation*}
\d(\A,\C) > \d(\A,\B)\,.
\end{equation*}
Thus, geometric in-betweenness does not follow from the order relation alone, but depends on the compatibility between the metric \(\d\) and the order structure of \(\mathbb{P}_n\). This leads naturally to the question when does the order relation imply the geometric relation.

\noindent This question has been studied in the setting of partially ordered metric spaces, where such spaces are said to be radially convex (see \cite{Ok2025} and the references therein). In view of the broader types of spaces that arise in matrix analysis and operator theory, we adopt the following more general definition.

\begin{definition}[Radial Convexity]

Let $\mathcal{X}$ be a set equipped with a binary relation $\preceq$, and let
\begin{equation*}
\d: \mathcal{X} \times \mathcal{X}\longrightarrow[0,\infty)
\end{equation*}
be a nonnegative function. We say that $\d$  is radially convex with respect to $\preceq$ if
\begin{equation*}
\A \preceq \C \preceq \B
\end{equation*}
implies
\begin{equation*}
\d(\A, \C) \leq \d(\A, \B) \qquad \text{ and } \qquad \d(\C, \B) \leq \d(\A, \B)\,.
\end{equation*}
\end{definition}

\begin{remark}

In the above definition, $\preceq$ is not required to be a partial order and $\d$ is not required to be a metric. This additional generality is useful in matrix analysis, where several natural relations are order-like without being partial orders, and where nonnegative divergences frequently arise in place of metrics.

\noindent Moreover, if $\D = \d^{2}$, then radial convexity with respect to $\d$ is equivalent to radial convexity with respect to $\D$, since the function $t\mapsto t^{2}$ is increasing on $\left[0, \infty\right)$. This observation will allow us to work interchangeably with a metric and its square when convenient.

\end{remark}

The paper is organized as follows. Section 2 contains the necessary preliminaries on the Loewner order, matrix monotone functions, and Loewner matrices. In Section 3, we establish radial convexity for several natural metrics on $\mathbb{P}_n$. We then show that matrix monotonicity provides a general mechanism for radial convexity. More precisely, if $f$ is matrix monotone of degree $n$ and $\left\|\cdot\right\|$ is a unitarily invariant norm, then the metric
\begin{equation*}
d_{\left\|\cdot\right\|,f}(\A,\B) = \left\|f(\A)-f(\B)\right\|
\end{equation*}
is radially convex with respect to the Loewner order on $\mathbb{P}_{n}$.

\noindent This provides a large family of metrics for which order-theoretic in-betweenness implies metric in-betweenness, and includes several classical geometries on the cone of positive definite matrices.

\medskip

\noindent In Section 4, we investigate the converse implication, concentrating on the metric induced by the Frobenius norm. In this case, the Hilbert space structure of the Frobenius norm allows us to obtain a characterization of radial convexity. From this point on, since the Frobenius norm is fixed, we simply say that $f$ is radially convex of order $n$ if the metric induced by $f$ is radially convex with respect to the Loewner order on $\mathbb{P}_{n}$. More precisely, radial convexity implies that
\begin{equation*}
\Re\left\langle\D f(\A)(\X),\D f(\A)(\Y)\right\rangle_{\F} \geq 0
\end{equation*}
for every $\A \in \mathbb{P}_{n}$ and every $\X, \Y \geq 0$. Using the Daleckii-Krein formula
\begin{equation*}
\D f(\A)(\H) = \L_{f}(\A)\circ\H\,,
\end{equation*}
we translate this differential condition into a condition on the Loewner matrix $\L_{f}(\A)$ (see Definition \ref{DefinitionLoewner}). In particular, radial convexity implies
\begin{equation*}
\L_{f}(\A) \circ \L_{f}(\A) \geq 0
\end{equation*}
for every $\A\in\mathbb{P}_{n}$, where $\circ$ denotes the Hadamard product. Thus, the converse problem is closely related to the question of when positivity of the Hadamard square of a Loewner matrix forces positivity of the Loewner matrix itself.

\noindent The results obtained in this paper for the Frobenius norm are summarized in the following diagram.

\bigskip

\begin{center}
\begin{tikzcd}[
column sep=large, row sep=large, cells={nodes={draw, text width=4.2cm, minimum height=1.2cm, align=center, inner sep=3pt}}]

\substack{f \text{ is increasing and radially } \\ \text{ convex of order } n}
\arrow[d, Rightarrow, shorten <=3pt, shorten >=3pt]
\arrow[r, Leftarrow, shorten <=3pt, shorten >=3pt]
&
\substack{f \text{ is increasing and matrix } \\ \text{ monotone of degree } n}
\arrow[d, Leftrightarrow, shorten <=3pt, shorten >=3pt] \\

\substack{\langle \D f(\A)(\X)\,,\D f(\A)(\Y)\rangle_{\F} \geq 0 \\ \left(\forall \A \in \mathbb{P}_{n}\right) \thinspace \left(\forall \X, \Y \in \mathbb{H}^{+}_{n}\right)}
\arrow[d, Leftrightarrow, shorten <=3pt, shorten >=3pt]
\arrow[r, Leftarrow, shorten <=3pt, shorten >=3pt]
&
\substack{\D f(\A)(\H) \geq 0 \\ \left(\forall \A \in \mathbb{P}_{n}\right) \thinspace \left(\forall  \H \in \mathbb{H}^{+}_{n}\right)}
\arrow[d, Leftrightarrow, shorten <=3pt, shorten >=3pt] \\

\substack{\L_{f}(\A) \circ \L_{f}(\A) \geq 0 \\ \left(\forall \A \in \mathbb{P}_{n}\right)}
\arrow[r, Leftarrow, shorten <=3pt, shorten >=3pt]
&
\substack{\L_{f}(\A) \geq 0 \\ \left(\forall \A \in \mathbb{P}_{n}\right)}
\end{tikzcd}
\end{center}

\bigskip

\noindent These results lead us to the following conjecture. 

\begin{conjecture*}

Let $f: \left(0, \infty\right) \to \mathbb{R}$. Then $f$ is radially convex of order $n$ (with respect to the Frobenius norm) if and only if $f$ is matrix monotone of degree $n$\,.

\end{conjecture*}

\noindent We prove this in degree $2$. Indeed, in this case positivity of
$\L_f(\A) \circ \L_f(\A)$ forces positivity of $\L_f(\A)$, and hence radial convexity of order $2$ is equivalent to matrix monotonicity of degree $2$. In degree $3$, we show that radial convexity already provides the principal minor conditions required for matrix monotonicity, leaving the sign of the determinant of the Loewner matrix as the remaining obstruction. We also give explicit higher-dimensional examples supporting the conjecture. Finally, Section 5 contains concluding remarks and investigates to what extent the degree $2$ result persists for arbitrary unitarily invariant norms.

\section{Preliminaries}

In this section, we recall some basic facts concerning Hermitian matrices, the Loewner order, and operator/matrix monotone functions. Throughout this paper, we denote by $\Mat(n) := \Mat(n, \mathbb{C})$ the algebra of $n\times n$ matrices with complex entries.

\noindent We denote by $\U_{n}$ the unitary group
\begin{equation*}
\U_{n} := \left\{g \in \Mat(n)\,, gg^{*} = \Id_{n}\right\},
\end{equation*}
and by $\mathbb{H}_{n}$ the real vector space of Hermitian matrices of size $n$,
\begin{equation*}
\mathbb{H}_{n} := \left\{\X \in \Mat(n)\,, \X^{*} = \X\right\}.
\end{equation*}

\noindent Recall that for every $\A \in \mathbb{H}_{n}$, there exists a unitary matrix $\U \in \U_{n}$ and real numbers $\lambda_{1}\,, \ldots\,, \lambda_{n}$ such that
\begin{equation*}
\A = \U\diag(\lambda_{1}\,,\ldots\,,\lambda_{n})\U^{*},
\end{equation*}
so that the spectrum of a Hermitian matrix is contained in $\mathbb{R}$. Similarly, we denote by $\mathbb{H}^{+}_{n}$ and $\mathbb{P}_{n}$ the subsets of $\mathbb{H}_{n}$ given by
\begin{equation*}
\mathbb{H}^{+}_{n} := \left\{\X \in \mathbb{H}_{n}\,, \X \geq 0\right\}\,, \qquad  \mathbb{P}_{n} := \left\{\X \in \mathbb{H}_{n}\,, \X > 0\right\}\,.
\end{equation*}

\begin{definition}

Let $\A\,, \B \in \mathbb{P}_{n}$. We say that $\A$ is smaller than or equal to $\B$ with respect to the Loewner order, and write $\A \leq \B$, if $\B - \A \in \mathbb{H}^{+}_{n}$\,. 

\end{definition}

\begin{remark}

Similarly, we write $\A<\B$ whenever $\B-\A \in \mathbb{P}_{n}$. Notice that the Loewner order $\leq$ is defined on the entire space $\mathbb{H}_{n}$, although in this paper we will be primarily interested in its restriction to the positive definite cone $\mathbb{P}_{n}$\,.

\end{remark}

\noindent Let $f:\I\to\mathbb{R}$ be a continuous function, where $\I$ is a subinterval of $(0,\infty)$. If $\A \in \mathbb{H}_{n}$ satisfies $\Spec(\A) \subseteq \I$, the functional calculus allows us to define $f(\A)$. Indeed, if
\begin{equation*}
\A = \U\diag(\lambda_{1}\,,\ldots\,,\lambda_{n})\U^{*}\,,
\end{equation*}
then
\begin{equation*}
f(\A) := \U\diag\left(f(\lambda_{1})\,,\ldots\,,f(\lambda_{n})\right)\U^{*}\,.
\end{equation*}
This definition does not depend on the choice of the unitary diagonalization of $\A$. We now recall the notion of an operator monotone function\,.

\begin{definition}

Let $\I \subseteq \mathbb{R}$ be an interval. A function $f: \I \longrightarrow \mathbb{R}$ is said to be operator monotone increasing if, for every $n \geq 1$ and every $\A\,, \B \in \mathbb{H}_{n}$ such that $\Spec(\A)\,, \Spec(\B) \subseteq \I$ and $\A \leq \B$, we have
\begin{equation*}
f(\A)\leq f(\B)\,.
\end{equation*}
Similarly, $f$ is said to be operator monotone decreasing if
\begin{equation*}
\A \leq \B \qquad \Longrightarrow \qquad f(\A) \geq f(\B)
\end{equation*}
for every $n \geq 1$ and every $\A\,, \B \in \mathbb{H}_{n}$ with spectra contained in $\I$\,.

\end{definition}

\begin{remark}

Clearly, $f$ is operator monotone decreasing if and only if $-f$ is operator monotone increasing. When no confusion can arise, we will simply say that $f$ is operator monotone if it is operator monotone in one of these two directions\,.

\end{remark}

\noindent In this paper, we will mainly consider functions defined on the positive half-line
\begin{equation*}
f: \left(0, \infty\right) \longrightarrow \mathbb{R}\,,
\end{equation*}
and hence the functional calculus will be applied to matrices in $\mathbb{P}_{n}$ for all $n$\,.

\begin{notation}

For a continuously differentiable function $f: \I \longrightarrow \mathbb{R}$, and $\A \in \mathbb{P}_{n}$ such that $\Spec(\A) \subseteq \I$, we denote by
\begin{equation*}
\D f(\A): \mathbb{H}_{n} \longrightarrow \mathbb{H}_{n}
\end{equation*}
the Fr\'echet derivative of the matrix function induced by $f$ at $\A$, i.e.
\begin{equation*}
\D f(\A)(\H) = \lim\limits_{t \to 0} \frac{f(\A + t\H) - f(\A)}{t}\,, \qquad \left(\H \in \mathbb{H}_{n}\right)\,.
\end{equation*}

\end{notation}

\noindent We recall a differential characterization of operator monotone functions (see \cite{BHATIA})\,.

\begin{proposition}

Let $f: \I \to \mathbb{R}$ be continuously differentiable. Then $f$ is operator monotone increasing if and only if, for every $n\geq1$, every $\A \in \mathbb{P}_{n}$ with $\Spec(\A) \subseteq \I$, and every $\H \in \mathbb{H}^{+}_{n}$, 
\begin{equation*}
\D f(\A)(\H)\geq0\,.
\end{equation*}
Similarly, $f$ is operator monotone decreasing if and only if
\begin{equation*}
\D f(\A)(\H) \leq 0
\end{equation*}
for every $n \geq 1$, every $\A \in \mathbb{P}_{n}$ with $\Spec(\A) \subseteq \I$, and every $\H \in \mathbb{H}^{+}_{n}$\,.

\label{PropositionMonotoneOne}

\end{proposition}

\begin{proof}

See \cite{BHATIA} (see also \cite{BHATIASANO})\,.

\end{proof}

\begin{lemma}

Let $\A\in\mathbb{P}_n$ and suppose that $\A=\P\Delta\P^*$, where $\Delta$ is diagonal and $\P\in\U_n$.
Then
\begin{equation*}
\D f(\A)(\H) = \P \D f(\Delta)(\P^{*}\H\P)\P^{*}\,, \qquad \left(\H \in \mathbb{H}^{+}_n\right).
\end{equation*}
In particular, $\D f(\A)(\H) \geq 0$ for all $\H \in \mathbb{H}^{+}_{n}$ if and only if $\D f(\Delta)(\X) \geq 0$ for all $\X \in \mathbb{H}^{+}_{n}$\,.

\label{LemmaAtoD}

\end{lemma}

\begin{proof}

We have
\begin{eqnarray*}
\D f (\A)(\H) & = & \lim\limits_{t \to 0} \frac{f(\A+t\H) - f(\A)}{t} = \lim\limits_{t \to 0} \frac{f(\A+t\H) - f(\A)}{t} \\ 
& = & \lim\limits_{t \to 0} \frac{f(\P\Delta\P^{*}+t\H) - f(\P\Delta\P^{*})}{t} \\
& = & \lim\limits_{t \to 0} \frac{f(\P\Delta\P^{*}+t\P\P^{*}\H\P\P^{*}) - f(\P\Delta\P^{*})}{t} \\ 
& = & \P\left(\lim\limits_{t \to 0} \frac{f(\Delta + t\P^{*}\H\P) - f(\Delta)}{t}\right)\P^{*} \\
& = & \P\D f(\Delta)(\P^{*}\H\P)\P^{*}\,.
\end{eqnarray*}

\end{proof}

\begin{corollary}

The function $f$ is operator monotone increasing if and only if, for every $n \in \mathbb{N}$, every diagonal matrix $\Delta \in \mathbb{P}_{n}$, and every $\H \in \mathbb{H}^{+}_{n}$, we have
\begin{equation*}
\D f(\Delta)(\H) \geq 0\,.
\end{equation*}

\end{corollary}

\noindent We now recall the definition of the Loewner matrix associated with a differentiable function $f:(0,\infty)\to\mathbb{R}$ and a diagonal matrix $\D=\diag(\lambda_i)\in\mathbb{P}_n$ (see \cite{BHATIA}). For all $x, y \in \left(0\,, \infty\right)$, we define the first divided difference of $f$ by
\begin{equation*}
f^{[1]}(x,y) = \begin{cases} \displaystyle\frac{f(x)-f(y)}{x-y} & \text{ if } x \neq y \\  f'(x) & \text{ if } x=y
\end{cases}
\end{equation*}

\begin{definition}

For $f$ as above and $\D=\diag(\lambda_1,\ldots,\lambda_n)\in\mathbb{P}_n$, the Loewner matrix associated with $f$ and $\D$ is defined by
\[
\L_f(\D)=\left(f^{[1]}(\lambda_i,\lambda_j)\right)_{1\leq i,j\leq n}.
\]

\label{DefinitionLoewner}

\end{definition}

\noindent If the eigenvalues $\lambda_1,\ldots,\lambda_n$ are pairwise distinct, then
\begin{equation*}
\L_{f}(\D) = \begin{pmatrix} f'(\lambda_{1}) & \displaystyle\frac{f(\lambda_{1})-f(\lambda_{2})}{\lambda_{1}-\lambda_{2}} & \cdots & \displaystyle\frac{f(\lambda_{1})-f(\lambda_{n})} {\lambda_{1}-\lambda_{n}} \\ \displaystyle\frac{f(\lambda_{2})-f(\lambda_{1})}{\lambda_{2}-\lambda_{1}} & f'(\lambda_{2}) & \cdots & \displaystyle\frac{f(\lambda_{2})-f(\lambda_{n})}{\lambda_{2}-\lambda_{n}} \\ \vdots & \vdots & \ddots & \vdots \\
\displaystyle\frac{f(\lambda_{n})-f(\lambda_{1})}{\lambda_{n}-\lambda_{1}} & \displaystyle\frac{f(\lambda_{n})-f(\lambda_{2})}{\lambda_{n}-\lambda_{2}} &
\cdots & f'(\lambda_{n}) \end{pmatrix}\,.
\end{equation*}

\begin{remark}

Let $\A \in \mathbb{P}_{n}$. By the spectral theorem, there exist $\U \in \U_n$ and a diagonal matrix $\D = \diag\left(\lambda_{i}\right)$, with $\lambda_{i} > 0$, such that $\A=\U\D\U^{*}$. We will write 
\begin{equation*} 
\L_{f}(\A) := \L_{f}(\D)\,. 
\end{equation*} 
This notation depends on the ordering of the eigenvalues of $\A$. However, a different ordering only produces a simultaneous permutation of the rows and columns of $\L_{f}(\A)$. In particular, properties such as 
\begin{equation*} 
\L_{f}(\A) \geq 0 \qquad \text{ or } \qquad  \L_{f}(\A) \circ \L_{f}(\A) \geq 0 
\end{equation*} 
are independent of the ordering of the eigenvalues\,. 

\label{LoewnerMatrixA}

\end{remark}

\begin{remark}

Let $\A = \left(a_{i,j}\right)$ and $\B = \left(b_{i,j}\right)$ be two matrices in $\Mat(n)$. The Hadamard product of $\A$ and $\B$, denoted by $\A \circ \B$, is defined by
\begin{equation*}
\A \circ \B = \left(a_{i,j}b_{i,j}\right)_{i,j}\,.
\end{equation*}
We recall some elementary properties of this product. For matrices of $\Mat(n)$, the Hadamard product is commutative, associative, and
bilinear. Moreover,
\begin{equation*}
\left(\A \circ \B\right)^{*} = \A^{*} \circ \B^{*}\,,
\end{equation*}
so the Hadamard product of two Hermitian matrices is again Hermitian\,.

\noindent If $\J$ denotes the matrix whose entries are all equal to $1$, then
\begin{equation*}
\A \circ \J = \A\,,
\end{equation*}
i.e. $\J$ plays the role of the identity matrix for the Hadamard product.

\noindent It is known that if $\A \geq 0$ and $\B \geq 0$, then
\begin{equation*}
\A \circ \B \geq 0\,.
\end{equation*}
In particular,
\begin{equation*}
\A \geq 0 \qquad \Longrightarrow \qquad \A \circ \A \geq 0\,.
\end{equation*}
Note that the same conclusion holds if $\A \leq 0$, since
\begin{equation*}
\A \circ \A = \left(-\A\right) \circ \left(-\A\right) \geq 0\,.
\end{equation*}
Similarly, if $\A \geq 0$ and $\B \leq 0$, then $\A \circ \B \leq 0$\,.

\label{RemarkHadamard}

\end{remark}

\noindent The Fr\'echet derivative $f$ of a matrix function admits a nice expression in terms of the corresponding Loewner matrix $\L_{f}$. The following result is commonly referred to as the Daleckii-Krein formula\,.

\begin{theo}

Let $f: \left(0, \infty\right) \to \mathbb{R}$ be continuously differentiable and let $\A = \diag\left(\lambda_{i}\right) \in \mathbb{P}_{n}$. Then, for every $\H \in \mathbb{H}_{n}$,
\begin{equation*}
\D f(\A)(\H) = \L_{f}(\A) \circ \H\,.
\end{equation*}

\label{TheoremDifferentialLoewner}

\end{theo}

\begin{remark}

More generally, if $\A = \U\Delta\U^{*}$, with $\U \in \U_{n}$ and $\Delta = \diag\left(\lambda_{i}\right), \lambda_{i} > 0$, then it follows from Lemma \ref{LemmaAtoD} and Remark \ref{LoewnerMatrixA} that
\begin{equation*}
\D f(\A)(\H) = \U\left[\L_{f}(\A) \circ \left(\U^{*}\H\U\right)\right]\U^{*}\,.
\end{equation*}

\end{remark}

\noindent Combining the characterization of operator monotonicity given in Proposition \ref{PropositionMonotoneOne} with Theorem \ref{TheoremDifferentialLoewner} and Remark \ref{RemarkHadamard} gives the following characterization in terms of the Loewner matrix\,.

\begin{theo}

Let $\I\subseteq\mathbb{R}$ be an interval and let $f: \I \longrightarrow \mathbb{R}$ be continuously differentiable. Then $f$ is operator monotone increasing if and only if, for every $n \geq 1$ and every $\A \in \mathbb{P}_{n}$ such that $\Spec(\A) \subseteq \I$,
\begin{equation*}
\L_{f}(\A) \geq 0\,.
\end{equation*}
Similarly, $f$ is operator monotone decreasing if and only if, for every $n \geq 1$ and every $\A \in \mathbb{P}_{n}$ (with $\Spec(\A) \subseteq \I$),
\begin{equation*}
\L_{f}(\A) \leq 0\,.
\end{equation*}

\label{TheoremLoewner}

\end{theo}

\noindent The previous characterization of operator monotone functions requires the positivity of the Loewner matrix $\L_{f}(\A)$ for every matrix size $n\geq 1$. A natural weaker notion arises by fixing the dimension and requiring the monotonicity condition only for matrices of a given size (introduced in \cite{BENDATSHERMAN}). In particular, we get the following definition. 

\begin{definition}

Let $n \geq 1$ and let $f: \I \longrightarrow \mathbb{R}$ be a function defined on an interval $\I \subseteq \mathbb{R}$. We say that $f$ is matrix monotone increasing of degree $n$ if, for every $\A, \B \in \mathbb{H}_{n}$ with spectrum contained in $\I$,
\begin{equation*}
\A \leq \B \qquad \Longrightarrow \qquad f(\A) \leq f(\B)\,.
\end{equation*}
Similarly, $f$ is said to be matrix monotone decreasing of degree $n$ if
\begin{equation*}
\A \leq\B \qquad \Longrightarrow \qquad f(\A)\geq f(\B)\,.
\end{equation*}

\end{definition}
When no confusion can arise, we say that $f$ is matrix monotone of degree $n$
if it is matrix monotone in either of these two directions.
\begin{proposition}

Let $\I\subseteq\mathbb{R}$ be an interval and let $f:\I\to\mathbb{R}$ be matrix monotone of degree $n$. Then $f$ is matrix monotone of degree $m$ for every $1\leq m\leq n$.
\label{MonotoneFromNtoM}

\end{proposition}

\begin{proof}
Without loss of generality, assume that $f$ is matrix monotone increasing; otherwise, replace $f$ by $-f$.
Let $1\leq m<n$, and let $\A,\B\in\mathbb{H}_m$ have spectra contained in $\I$ and satisfy $\A\leq\B$. Choose $\alpha\in\I$. We denote by $\widetilde{\A}$ and $\widetilde{\B}$ the matrices in $\Mat(n)$ given by
\begin{equation*}
\widetilde{\A} := \A \oplus \alpha \Id_{n-m} = \begin{pmatrix} \A & 0 \\ 0 & \alpha\Id_{n-m} \end{pmatrix}\,,  \qquad \widetilde{\B} := \B \oplus \alpha \Id_{n-m}\,.
\end{equation*}
We have $\widetilde{\A}\,, \widetilde{\B} \in \mathbb{H}_{n}$ and
\begin{equation*}
\widetilde{\B} - \widetilde{\A} = \begin{pmatrix} \B - \A & 0 \\ 0 & 0 \end{pmatrix}\,,
\end{equation*}
i.e. $\widetilde{\A} \leq \widetilde{\B}$. Since $f$ is matrix monotone of degree $n$, it follows that
\begin{equation*}
f(\widetilde{\A}) \leq f(\widetilde{\B})\,.
\end{equation*}
Moreover, it follows from functional calculus that
\begin{equation*}
f(\widetilde{\A}) = f(\A)\oplus f(\alpha)\Id_{n-m} \qquad \text{ and } \qquad f(\widetilde{\B}) = f(\B) \oplus f(\alpha)\Id_{n-m}\,.
\end{equation*}
Therefore,
\begin{equation*}
f(\A) \oplus f(\alpha)\Id_{n-m} \le	 f(\B) \oplus f(\alpha)\Id_{n-m}\,,
\end{equation*}
which implies
\begin{equation*}
f(\A) \leq f(\B)\,.
\end{equation*}
Hence $f$ is matrix monotone of degree $m$\,.

\end{proof}

\begin{remark}

We have a natural analogue of Proposition \ref{PropositionMonotoneOne} and Theorem \ref{TheoremLoewner} for matrix monotone $f$ of degree $n$. Indeed, a continuously differentiable function $f: \I \to \mathbb{R}$ is matrix monotone increasing of degree $n$ if and only if for every $\A \in \mathbb{P}_{n}$ with $\Spec(\A) \subseteq \I$, and every $\H \in \mathbb{H}^{+}_{n}$, 
\begin{equation*}
\D f(\A)(\H) \geq 0\,.
\end{equation*}
In particular, $f$ is matrix monotone increasing of degree $n$ if and only if $\L_{f}(\A) \geq 0$ for all $\A \in \mathbb{P}_{n}$\,.

\label{RemarkMatrixMonotone}

\end{remark}

\noindent However, the previous distinction is genuine: a function may preserve the Loewner order for matrices of degree $n$, but not $n+1$. The following example illustrates this phenomenon and motivates the notion of matrix monotonicity of finite degree (see \cite{HANSENJITOMIYAMA})\,.

\begin{example}

Consider the function
\begin{equation*}
f_{2}(x) = x + \frac{x^{3}}{3}\,.
\end{equation*}
Since
\begin{equation*}
f'_{2}(x)=1+x^{2}>0
\end{equation*}
for every $x>0$, the function $f_{2}$ is strictly increasing, and hence matrix monotone of degree $1$ on every interval contained in $\left(0\,, \infty\right)$.

\noindent For all $\D = \diag\left(\lambda_{1}\,, \lambda_{2}\right)$, with $\lambda_{1} \neq \lambda_{2}$, we have
\begin{equation*}
\Omega(\lambda) := \frac{f_{2}(\lambda_{1})-f_{2}(\lambda_{2})} {\lambda_{1}-\lambda_{2}} = 1+\frac{\lambda_{1}^{2}+\lambda_{1}\lambda_{2} +\lambda_{2}^{2}}{3}\,.
\end{equation*}
Thus the determinant of the corresponding Loewner matrix $\L_{f_{2}}(\D)$ is
\begin{eqnarray*}
f'_{2}(\lambda_{1})f'_{2}(\lambda_{2})-\Omega(\lambda)^{2} & = & \left(1+\lambda_{1}^{2}\right)\left(1+\lambda_{2}^{2}\right) - \left(1+\frac{\lambda_{1}^{2}+\lambda_{1}\lambda_{2}+\lambda_{2}^{2}}{3} \right)^{2} \\
& = & -\frac{\left(\lambda_{1}-\lambda_{2}\right)^{2}\left(\lambda_{1}^{2} + 4\lambda_{1}\lambda_{2}+\lambda_{2}^{2}-3\right)}{9}\,.
\end{eqnarray*}
In particular, the Loewner matrix $\L_{f_{2}}(\D)$ is positive semidefinite when
\begin{equation*}
\lambda_{1}^{2} + 4\lambda_{1}\lambda_{2} + \lambda_{2}^{2} \leq 3\,,
\end{equation*}
i.e. this holds for every 
\begin{equation*}
0 < \lambda_{1}\,,\lambda_{2} < \frac{1}{\sqrt{2}}\,.
\end{equation*}
Therefore $f_{2}$ is matrix monotone of degree $2$ on $\I = \left(0\,, \frac{1}{\sqrt{2}}\right)$. Using the Hankel matrix, the authors proved in \cite{HANSENJITOMIYAMA} that $f_{2}$ is not matrix monotone of degree 3. 

\noindent Indeed, assume that $f_{2}$ is matrix monotone of degree $3$ on an interval $\I$. Then for every $x \in \I$, the associated Hankel matrix
\begin{equation*}
\H_{3}(f_{2}, x) := \left(\frac{f^{(i+j-1)}_{2}(x)}{(i+j-1)!}\right)_{1\leq i, j \leq 3}
\end{equation*}
must be positive semidefinite. For $f_{2}(x)=x+\frac{x^{3}}{3}$, we have
\begin{equation*}
f'_{2}(x)=1+x^{2}\,, \qquad f''_{2}(x)=2x\,, \qquad f'''_{2}(x)=2\,, \qquad f^{(4)}_{2}(x) = f^{(5)}_{2}(x) = 0\,, 
\end{equation*}
so
\begin{equation*}
\H_{3}(f_{2}, x) = \begin{pmatrix} 1+x^{2} & x & \frac{1}{3} \\ x & \frac{1}{3} & 0 \\ \frac{1}{3} & 0 & 0 \end{pmatrix}\,.
\end{equation*}
A direct computation gives
\begin{equation*}
\det \H_{3}(f_{2}, x) = -\frac{1}{27} < 0\,.
\end{equation*}
Therefore $\H_{3}(f_{2}, x)$ is not positive semidefinite. In particular, composing $f_{2}$ with the operator monotone function
\begin{equation*}
\varphi(x) = \frac{1}{\sqrt{2}}\frac{x}{x+1}\,,
\end{equation*}
we obtain a function $\widetilde{f}_{2} = f_{2} \circ \varphi$ which is matrix monotone of degree $2$ on $\left(0, \infty\right)$, but not of degree $3$.

\noindent Similarly, one can show that the function 
\begin{equation*}
f_{3}(x) = x + \frac{x^{3}}{3} + \frac{x^{5}}{5}
\end{equation*}
is matrix monotone of degree $3$ on $\left(0, 0.2\right)$, so the function $\widetilde{f}_{3} = f_{3} \circ \widetilde{\varphi}$ is matrix monotone of degree $3$ on $\left(0, \infty\right)$, but not of degree $4$, with $\widetilde{\varphi} = 0.2\sqrt{2}\varphi$\,.

\label{ExampleMonotone}

\end{example}

\section{Important Examples of Radial Convexity}

\label{SectionMtoRC}

\noindent In this section, we present several examples of metrics on the cone $\mathbb{P}_{n}$ that are radially convex with respect to the Loewner order, namely the Thompson metric, the standard Riemannian metric, and the Bures-Wasserstein metric. We start with the Thompson metric \cite{Thompson1963}. 

\begin{definition}[Thompson Metric]
Let $\A\,, \B \in \mathbb{P}_{n}$. We define
\begin{equation*}
\M(\A/\B) = \inf\left\{\alpha>0\,, \A \leq \alpha \B\right\}\,.
\end{equation*}
The Thompson metric on $\mathbb{P}_{n}$ is defined by
\begin{equation*}
\d_{\infty}(\A,\B) := \log\max\left\{\M(\A/\B)\,,\M(\B/\A)\right\}\,.
\end{equation*}

\end{definition}

\begin{remark}

The Thompson metric is naturally related to the Loewner order. Indeed, for $\A, \B \in \mathbb{P}_{n}$, one has
\begin{equation*}
\M(\A/\B) = \lambda_{\max}\left(\B^{-\frac{1}{2}}\A\B^{-\frac{1}{2}}\right)\,.
\end{equation*}
Therefore,
\begin{equation*}
\d_{\infty}(\A,\B) = \max\left\{\log\lambda_{\max}\left(\B^{-\frac{1}{2}}\A\B^{-\frac{1}{2}}\right)\,, \log\lambda_{\max}\left(\A^{-\frac{1}{2}}\B\A^{-\frac{1}{2}}\right)\right\}\,.
\end{equation*}
Equivalently,
\begin{equation*}
\d_{\infty}(\A, \B) = \left\|\log\left(\A^{-\frac{1}{2}}\B\A^{-\frac{1}{2}}\right)\right\|_{\infty}\,,
\end{equation*}
where $\left\| \cdot \right\|_{\infty}$ denotes the operator norm. Moreover, $\d_{\infty}$ satisfies
\begin{equation*}
\d_{\infty}\left(\S\A\S^{*}, \S\B\S^{*} \right) = \d_{\infty}(\A,\B)\,, \qquad \d_{\infty}(\alpha\A, \alpha\B) = \d_{\infty}(\A,\B)\,,
\end{equation*}
for any invertible matrix $\S$ and $\alpha > 0$.

\noindent Finally, for $r\geq0$,
\begin{equation*}
\d_{\infty}(\A,\B) \leq r \qquad \text{ if and only if } \qquad  e^{-r}\A \leq \B \leq e^{r}\A\,.
\end{equation*}

\label{RemarkThompson}

\end{remark}

\begin{proposition} 
The Thompson metric is radially convex with respect to the Loewner order on
$\mathbb{P}_{n}$\,.

\end{proposition} 

\begin{proof} 

Let $\A\,, \B\,, \C \in \mathbb{P}_{n}$ such that $\A \leq \C \leq \B$,  and set 
\begin{equation*} 
r = \d_{\infty}(\A, \B)\,. 
\end{equation*} 
It follows from Remark \ref{RemarkThompson} that 
\begin{equation*} 
e^{-r}\A \leq \B \leq e^{r}\A\,. 
\end{equation*} 
Since $\A \leq \C \leq \B$, we obtain 
\begin{equation*} 
e^{-r}\A \leq \A \leq \C \leq \B \leq e^{r}\A\,. 
\end{equation*} 
Therefore, 
\begin{equation*} 
e^{-r}\A \leq \C \leq e^{r}\A\,, 
\end{equation*} 
and Remark \ref{RemarkThompson} implies that 
\begin{equation*} 
\d_{\infty}(\A, \C) \leq r\,. 
\end{equation*} 
Moreover, since $\B\leq e^{r}\A$, we have 
\begin{equation*} 
e^{-r}\B \leq \A \leq \C\,. 
\end{equation*} 
Together with $\C \leq \B$, this gives 
\begin{equation*} 
e^{-r}\B \leq \C \leq \B \leq e^{r}\B\,. 
\end{equation*} 
Hence, once again by Remark \ref{RemarkThompson}, 
\begin{equation*} 
\d_{\infty}(\C,\B) \leq r\,. 
\end{equation*} 
Finally,
\begin{equation*} 
\d_{\infty}(\A,\B) \geq \max\left\{ \d_{\infty}(\A,\C)\,,\d_{\infty}(\C,\B) \right\}\,, 
\end{equation*} 
and the Thompson metric is radially convex with respect to the Loewner order.

\end{proof}

\noindent Similarly, the standard Riemannian metric on $\mathbb{P}_{n}$ is radially convex with respect to the Loewner order. We state this result without proof, as it follows from the operator monotonicity of the logarithm and the monotonicity of the Frobenius norm.

\begin{definition}

Let $\A\,, \B \in \mathbb{P}_{n}$. The standard Riemannian metric on $\mathbb{P}_{n}$ is defined by
\begin{equation*}
\d_{\R}(\A,\B) := \left\|\log\left(\A^{-\frac{1}{2}}\B\A^{-\frac{1}{2}}\right)\right\|_{2}\,,
\end{equation*}
where 
$\left\|\cdot\right\|_{2}$ denotes the Schatten $2$-norm (or Frobenius), i.e.
\begin{equation}
\|\A\|_{\F}:=\left\|\A\right\|_{2} = \left(\tr\left(\A^{*}\A\right)\right)^{\frac{1}{2}} = \left(\sum\limits_{i,j=1}^{n}\left|a_{i,j}\right|^{2}\right)^{\frac{1}{2}}\,, \qquad \left(\A = \left(a_{i,j}\right) \in \Mat(n)\right)\,.
\label{DefinitionFrobenius}
\end{equation}

\end{definition}

\begin{proposition}
The standard Riemannian metric on $\mathbb{P}_{n}$ is radially convex with respect to the Loewner order\,.
\end{proposition}

From this point on, we write $\|\cdot\|_{\F}$ for the Frobenius norm, while $\|\cdot\|$ denotes an arbitrary unitarily invariant norm unless stated otherwise.

\begin{definition}[Bures-Wasserstein Metric \cite{BhatiaJainLim2019}]

Let $\A, \B \in \mathbb{P}_{n}$. The Bures-Wasserstein metric on $\mathbb{P}_{n}$ is defined by
\begin{equation*}
\d_{\B\W}(\A,\B) := \left[\tr(\A)+\tr(\B) - 2\tr\left(\A^{\frac{1}{2}}\B\A^{\frac{1}{2}}\right)^{\frac{1}{2}}\right]^{\frac{1}{2}}\,.
\end{equation*}

\end{definition}

\begin{remark}

We recall some properties of the Bures-Wasserstein metric that will be useful in what follows. First, $\d_{\B\W}$ is symmetric, i.e.
\begin{equation*}
\d_{\B\W}(\A,\B) = \d_{\B\W}(\B,\A) \qquad \left(\A\,, \B \in \mathbb{P}_{n}\right)\,.
\end{equation*}
Moreover, the Bures-Wasserstein metric admits the following representation
\begin{equation*}
\d^{2}_{\B\W}(\A,\X) = \tr(\A) + \max\limits_{\Y>0}\left(\tr\left[\X\left(\Id_{n}-\Y^{-1}\right)-\A\Y\right]\right)\,.
\end{equation*}
For fixed $\A \in \mathbb{P}_{n}$, the maximum is attained at
\begin{equation*}
\Y_{\X} = \A^{-\frac{1}{2}}\left(\A^{\frac{1}{2}}\X\A^{\frac{1}{2}}\right)^{\frac{1}{2}}\A^{-\frac{1}{2}}\,.
\end{equation*}
In particular, the operator monotonicity of the square root gives
\begin{equation*}
\X \geq \A \qquad \Longrightarrow \qquad \Y_{\X} \geq \Id_{n} \qquad \Longrightarrow \qquad \Id_{n} - \Y^{-1}_{\X} \geq 0\,,
\end{equation*}
while
\begin{equation*}
\X \leq \A \qquad \Longrightarrow \qquad \Y_{\X} \leq \Id_{n} \qquad \Longrightarrow \qquad \Id_{n} - \Y^{-1}_{\X} \leq 0\,.
\end{equation*}

\label{RemarkBuresWasserstein}

\end{remark}

\begin{proposition}

The Bures-Wasserstein metric on $\mathbb{P}_{n}$ is radially convex with respect to the Loewner order\,.

\end{proposition}

\begin{proof}

Let $\A, \B, \C \in \mathbb{P}_{n}$ such that $\A \leq \C \leq \B$. To simplify the notation, for fixed $\Omega\in\mathbb{P}_n$, define
\begin{equation*}
g_{\Omega}(\X) := \max_{\Y>0}\left(\tr\left[\X\left(\Id_n-\Y^{-1}\right)-\Omega\Y\right]\right)\,.
\end{equation*}
In particular, it follows from Remark \ref{RemarkBuresWasserstein} that
\begin{equation*}
\d_{\B\W}^{2}(\A, \X) = \tr(\A) + g_{\A}(\X)\,.
\end{equation*}
Let $\Y_{\C}$ denote the maximizer in the definition of $g_{\A}(\C)$. Since $\C \geq \A$, Remark \ref{RemarkBuresWasserstein} gives
\begin{equation*}
\Y_{\C} \geq \Id_{n}\,,
\end{equation*}
and hence
\begin{equation*}
\Id_{n} - \Y^{-1}_{\C} \geq 0\,.
\end{equation*}
Since $\B \geq \C$, it follows that
\begin{equation*}
\tr\left[\B\left(\Id_{n} - \Y_{\C}^{-1}\right)\right] \geq \tr\left[\C\left(\Id_{n} - \Y_{\C}^{-1}\right)\right]\,.
\end{equation*}
Hence,
\begin{equation*}
g_{\A}(\B) \geq \tr\left[\B\left(\Id_{n} - \Y_{\C}^{-1}\right) - \A\Y_{\C}\right] \geq \tr\left[\C\left(\Id_{n} - \Y^{-1}_{\C}\right) - \A\Y_{\C}\right] = g_{\A}(\C)\,.
\end{equation*}
Thus
\begin{equation*}
\d^{2}_{\B\W}(\A, \C) \leq \d^{2}_{\B\W}(\A,\B)\,,
\end{equation*}
which implies
\begin{equation*}
\d_{\B\W}(\A, \C) \leq \d_{\B\W}(\A,\B)\,.
\end{equation*}
Similarly, using the same argument with $\B$ fixed and the fact that $\A \leq \C \leq \B$, we obtain
\begin{equation*}
\d_{\B\W}(\C, \B) \leq \d_{\B\W}(\A,\B)\,.
\end{equation*}
Finally,
\begin{equation*}
\d_{\B\W}(\A,\B) \geq \max\left\{\d_{\B\W}(\A, \C)\,, \d_{\B\W}(\C, \B)\right\}\,.
\end{equation*}

\end{proof}

The last example in this section, is the most important one for the objectives of this article. We now focus on the $f$-metric induced by a function $f:(0,\infty)\to (0,\infty)$ and a norm $\|\cdot\|,$
\begin{equation*}
\d_{\|\cdot\|, f}(\A, \B) := \left\|f(\A) - f(\B)\right\|\,.
\end{equation*}

\begin{proposition}\label{Prop:MonoConv}

Let $f:(0,\infty)\to\mathbb{R}$ be matrix monotone of degree $n$, and let $\|\cdot\|$ be a monotone norm on $\mathbb{P}_n$. Then $\d_{\|\cdot\|,f}$ is radially convex of order $n$ with respect to the Loewner order on $\mathbb{P}_n$.

\label{FromMonotoneToConvex}

\end{proposition}

\begin{proof}

Let $\A, \B\,, \C \in \mathbb{P}_{n}$ such that $\A \leq \C \leq \B$. Without loss of generality, we assume that $f$ is matrix monotone increasing of degree $n$; otherwise, replace $f$ by $-f$. Hence,
\begin{equation*}
0 \leq f(\C) - f(\A) \leq f(\B) - f(\A)\,.
\end{equation*}
Since $\left\|\cdot\right\|$ is monotone,
\begin{equation*}
\d_{\left\|\cdot\right\|, f}(\A, \C) = \left\|f(\C)-f(\A)\right\| \leq \left\|f(\B) - f(\A)\right\| = \d_{\left\|\cdot\right\|, f}(\A, \B)\,.
\end{equation*}
Similarly, 
\begin{equation*}
\d_{\left\|\cdot\right\|, f}(\C, \B) = \left\|f(\B)-f(\C)\right\| \leq \left\|f(\B) - f(\A)\right\| = \d_{\left\|\cdot\right\|, f}(\A, \B)\,.
\end{equation*}
Therefore
\begin{equation*}
\d_{\left\|\cdot\right\|,f}(\A, \B) \geq \max\left\{\d_{\left\|\cdot\right\|, f}(\A,\C)\,, \d_{\left\|\cdot\right\|,f}(\C, \B)\right\}\,,
\end{equation*}
and $\d_{\|\cdot\|,f}$ is radially convex of order $n$ with respect to the Loewner order.

\end{proof}
 
\section{Frobenius norm: From Radial Convexity to Monotonicity}

\label{SectionRCtoM}

In this section, we focus on radial convexity with respect to the Frobenius norm and investigate the converse direction of the last result obtained in the previous section. We first derive a necessary differential condition for radial convexity and then characterize this condition in terms of the positivity of the Hadamard square of the corresponding Loewner matrix. This leads naturally to the question of whether radial convexity of order $n$ implies matrix monotonicity of the same degree. We formulate this question as a conjecture, prove it in degree $2$, and conclude with examples supporting the conjecture in higher degrees.

\noindent Recall that, since the Frobenius norm is fixed throughout this section, we say that $f$ is radially convex of order $n$ if the metric $\d_f$ induced by $f$ is radially convex with respect to the Loewner order on $\mathbb{P}_n$. We start with a proposition\,.

\begin{proposition}

Let $f:(0,\infty)\to\mathbb{R}$ be continuous. If $f$ is radially convex of order $n$, then $f$ is monotone of degree $1$ on $(0,\infty)$\,.

\label{PropositionMonotoneDegreeOne}

\end{proposition}

\begin{remark}

We recall a classical result. Let $x\,, y\,, z \in\mathbb{R}$. Suppose that
\begin{equation*}
\left|x-z\right| \leq \left|x-y\right| \qquad \text{ and } \qquad \left|z-y\right| \leq \left|x-y\right|\,.
\end{equation*}
Then $z$ lies between $x$ and $y$. In other words, we have
\begin{equation*}
\min\left\{x\,, y\right\} \leq z \leq \max\left\{x\,, y\right\}\,.
\end{equation*}

\label{TechnicalRemark}

\end{remark}

\begin{proof}[Proof of Proposition \ref{PropositionMonotoneDegreeOne}]

Let $0 < a < c < b$. Let $\A, \B, \C$ be the matrices given by
\begin{equation*}
\A = a\Id_{n}\,, \qquad \C = c\Id_{n}\,, \qquad \B = b\Id_{n}\,.
\end{equation*}
Then $\A \leq \C \leq \B$. By the radial convexity hypothesis, we get
\begin{equation*}
\d(\A, \C) \leq \d(\A, \B) \qquad \text{ and } \qquad \d(\C, \B) \leq \d(\A, \B)\,.
\end{equation*}
Since
\begin{equation*}
f(a\Id_{n})=f(a)\Id_{n}\,, \qquad f(c\Id_{n})=f(c)\Id_{n}\,, \qquad f(b\Id_{n}) = f(b)\Id_{n}\,,
\end{equation*}
we have
\begin{equation*}
\d(\A,\C) = \left|f(a)-f(c)\right|\left\|\Id_{n}\right\| \qquad \text{ and } \qquad \d(\A, \B) = \left|f(a)-f(b)\right|\left\|\Id_{n}\right\|\,.
\end{equation*}
Thus,
\begin{equation*}
\left|f(a)-f(c)\right| \leq \left|f(a)-f(b)\right| \qquad \text{ and } \qquad \left|f(c) - f(b)\right| \leq \left|f(a)-f(b)\right|\,.
\end{equation*}
From Remark \ref{TechnicalRemark}, $f(c)$ lies between $f(a)$ and $f(b)$. Hence,
\begin{equation*}
\min\left\{f(a)\,, f(b)\right\} \leq f(c) \leq \max\left\{f(a)\,, f(b)\right\}\,.
\end{equation*}
Therefore, $f$ preserves betweenness on $\left(0\,, \infty\right)$, so $f$ is monotone of degree $1$\,.

\end{proof}

\noindent We now prove an analogue of Proposition \ref{MonotoneFromNtoM}

\begin{proposition}

Let $f: \left(0, \infty\right) \to \mathbb{R}$. If $f$ is radially convex of order  $n$ with respect to $\d := \d_{f}$, then it is radially convex of order $m$ for every $1 \leq m \leq n$\,.

\label{PropositionRCNToM}

\end{proposition}

\begin{remark}

Using the embedding $\Mat(m) \hookrightarrow \Mat(n)\,, 1 \leq m \leq n$, we get, for all $\X \in \mathbb{H}_{m}$, that
\begin{equation}
\left\|\X\oplus 0_{n-m}\right\| = \left\|\X\right\|\,.
\end{equation}
This is not specific to the Frobenius norm; the same result remains valid for any unitarily invariant norm\,.

\label{CompatibilityFrobeniusNorm}

\end{remark}

\begin{proof}

Let $1 \leq m <n$, and let $\A, \B, \C \in \mathbb{P}_{m}$ such that $\A \leq \C \leq \B$. Let $\alpha > 0$, and let
\begin{equation*}
\widetilde{\A} := \A \oplus \alpha\Id_{n-m}\,, \qquad \widetilde{\C} := \C \oplus \alpha\Id_{n-m}\,, \qquad \widetilde{\B} := \B \oplus \alpha\Id_{n-m}\,.
\end{equation*}
In particular, $\widetilde{\A} \leq \widetilde{\C} \leq \widetilde{\B}$. By radial convexity of order $n$, we get
\begin{equation*}
\d\left(\widetilde{\A}, \widetilde{\C}\right) \leq \d \left(\widetilde{\A}, \widetilde{\B}\right) \qquad \text{ and } \qquad \d\left(\widetilde{\C}, \widetilde{\B}\right) \leq \d\left(\widetilde{\A}, \widetilde{\B}\right)\,.
\end{equation*}
Again, by functional calculus, we get
\begin{equation*}
f(\widetilde{\A}) - f(\widetilde{\C}) = \left(f(\A) - f(\C)\right) \oplus 0_{n-m}\,, \qquad f(\widetilde{\A}) - f(\widetilde{\B}) = \left(f(\A) - f(\B)\right) \oplus 0_{n-m}\,,
\end{equation*}
and
\begin{equation*}
f(\widetilde{\C}) - f(\widetilde{\B}) = \left(f(\C) - f(\B)\right) \oplus 0_{n-m}\,.
\end{equation*}
Using the compatibility of the norm with block embeddings (see Remark \ref{CompatibilityFrobeniusNorm}), we obtain
\begin{equation*}
\left\|f(\A) - f(\C)\right\| = \d(\A,\C) \leq \d(\A,\B) = \left\|f(\A) - f(\B)\right\|
\end{equation*}
and
\begin{equation*}
\left\|f(\B) - f(\C)\right\|= \d(\B, \C) \leq \d(\A, \B) = \left\|f(\A) - f(\B)\right\|\,.
\end{equation*}
Thus $f$ is radially convex of order $m$.

\end{proof}

\begin{proposition}

Let $f: \left(0\,, \infty\right) \to \mathbb{R}$ be differentiable, and assume that $f$ is radially convex of order $n$. Then, for every $\A \in \mathbb{P}_{n}$ and every $\X\,, \Y \geq 0$, we have
\begin{equation*}
\Re\langle\D f(\A)(\X)\,,\D f(\A)(\Y)\rangle_{\F} \geq 0\,.
\end{equation*}
where $\langle\cdot\,, \cdot\rangle_{\F}$ is the inner product associated with the Frobenius norm $\left\|\cdot\right\|_{\F}$ (see Equation \eqref{DefinitionFrobenius}).

\label{PropositionSeptember20V1}

\end{proposition}

\begin{remark}

Since $\D f(\A)(\X)$ and $\D f(\A)(\Y)$ are Hermitian, their Frobenius inner product is real. Therefore, the real part in the previous proposition is unnecessary, and we may simply write
\begin{equation*}
\langle\D f(\A)(\X)\,,\D f(\A)(\Y)\rangle_{\F} \geq 0\,.
\end{equation*}

\end{remark}

\begin{proof}

Let $\A \in \mathbb{P}_{n}$ and let $\X\,, \Y \geq 0$. For all $s\,, t \geq 0$, we have
\begin{equation*}
\A \leq \A+t\X \leq \A+t\left(\X+s\Y\right)\,.
\end{equation*}
By radial convexity, we get
\begin{equation*}
\left\|f\left(\A+t\X\right) - f(\A)\right\|_{\F} \leq \left\|f\left(\A+t\left(\X+s\Y\right)\right)-f(\A)\right\|_{\F}\,.
\end{equation*}
In particular, we have
\begin{equation*}
\frac{\left\|f\left(\A+t\X\right) - f(\A)\right\|_{\F}}{t} \leq \frac{\left\|f\left(\A+t\left(\X+s\Y\right)\right)-f(\A)\right\|_{\F}}{t}\,,
\end{equation*}
and by letting $t \to 0^{+}$, we get
\begin{equation*}
\left\|\D f(\A)(\X)\right\|_{\F} \leq \left\|\D f(\A)\left(\X+s\Y\right)\right\|_{\F}\,.
\end{equation*}
Since $\D f(\A)$ is linear, we get that $\D f(\A)\left(\X+s\Y\right) = \D f(\A)(\X) + s\D f(\A)(\Y)$. By squaring both sides of the previous equation, we obtain
\begin{equation*}
\left\|\D f(\A)(\X)\right\|^{2}_{\F} \leq \left\|\D f(\A)(\X)+s\D f(\A)(\Y)\right\|^{2}_{\F}\,.
\end{equation*}
Using the Frobenius inner product, we get
\begin{equation*}
\left\|\D f(\A)(\X)+s\D f(\A)(\Y)\right\|^{2}_{\F} 
\end{equation*}
\begin{equation*}
= \left\|\D f(\A)(\X)\right\|^{2}_{\F} + 2s\Re\langle\D f(\A)(\X)\,, \D f(\A)(\Y)\rangle_{\F} + s^{2}\left\|\D f(\A)(\Y)\right\|^{2}_{\F}\,.
\end{equation*}
Hence
\begin{equation*}
2s\Re\langle\D f(\A)(\X)\,, \D f(\A)(\Y)\rangle_{\F} + s^{2}\left\|\D f(\A)(\Y)\right\|^{2}_{\F} \geq 0\,.
\end{equation*}
Dividing by $s > 0$ and letting $s \to 0^{+}$ gives
\begin{equation*}
\Re\langle\D f(\A)(\X)\,, \D f(\A)(\Y)\rangle_{\F} \geq 0\,.
\end{equation*}

\end{proof}

\noindent For a real symmetric matrix $\L$, we denote by $\Phi_{\L}$ the map on $\Mat(n)$ given by
\begin{equation*}
\Phi_{\L}(\H) = \L \circ \H\,, \qquad \left(\H \in \Mat(n)\right)\,.
\end{equation*}
Note that $\Phi_{\L}\left(\mathbb{H}_{n}\right) \subseteq \mathbb{H}_{n}$. Indeed, it follows from Remark \ref{RemarkHadamard} that for all $\X \in \mathbb{H}_{n}$, we get
\begin{equation*}
\Phi_{\L}(\X)^{*} = \left(\L \circ \X\right)^{*} = \L^{*} \circ \X^{*} = \L \circ \X = \Phi_{\L}(\X)\,.
\end{equation*}

\begin{lemma}

For all $\X\,, \Y \in \mathbb{H}_{n}$, we have
\begin{equation}
\langle\Phi_{\L}(\X)\,, \Y\rangle_{\F} = \langle\X\,, \Phi_{\L}(\Y)\rangle_{\F}\,.
\label{SelfAdjoint}
\end{equation}
In particular, we have
\begin{equation*}
\langle\Phi_{\L}(\X)\,, \Phi_{\L}(\Y)\rangle_{\F} = \langle\X\,, \Phi_{\L \circ \L}(\Y)\rangle_{\F}\,.
\end{equation*}

\label{LemmaAdjoint}

\end{lemma}

\begin{proof}

We have $\X \in \mathbb{H}_{n}$, so $\X = \X^{*}$ and $\Phi_{\L}(\X)^{*} = \Phi_{\L}(\X)$. Then
\begin{equation*}
\langle\Phi_{\L}(\X)\,, \Y\rangle_{\F} = \sum\limits_{i = 1}^{n} \left(\Phi_{\L}(\X)\Y\right)_{i,i} = \sum\limits_{i = 1}^{n}\left(\sum\limits_{j = 1}^{n}\Phi_{\L}(\X)_{i,j}\Y_{j,i}\right) = \sum\limits_{i = 1}^{n}\sum\limits_{j = 1}^{n} \L_{i,j}\X_{i,j}\Y_{j,i}
\end{equation*}
and
\begin{equation*}
\langle\X\,, \Phi_{\L}(\Y)\rangle_{\F} = \sum\limits_{i = 1}^{n} \left(\X\Phi_{\L}(\Y)\right)_{i,i} = \sum\limits_{i = 1}^{n}\left(\sum\limits_{j = 1}^{n} \X_{i,j}\Phi_{\L}(\Y)_{j,i}\right) = \sum\limits_{i = 1}^{n}\sum\limits_{j = 1}^{n} \X_{i,j}\L_{j,i}\Y_{j,i}\,.
\end{equation*}
Using that $\L = \L^{t}$, we get that $\L_{i,j} = \L_{j,i}$, and Equation \eqref{SelfAdjoint}. Moreover, using that
\begin{equation*}
\Phi_{\L}(\Phi_{\L}(\Y)) = \L \circ \Phi_{\L}(\Y) = \L \circ \left(\L \circ \Y\right) = \left(\L \circ \L\right) \circ \Y = \Phi_{\L \circ \L}(\Y)\,,
\end{equation*}
we get that
\begin{equation*}
\langle\Phi_{\L}(\X)\,, \Phi_{\L}(\Y)\rangle_{\F} = \langle\X\,, \Phi_{\L}(\Phi_{\L}(\Y))\rangle_{\F} = \langle\X\,, \Phi_{\L \circ \L}(\Y)\rangle_{\F}\,,
\end{equation*}
and the lemma follows\,.

\end{proof}

\begin{proposition}

Let  $f:(0,\infty)\to\mathbb{R}$ be differentiable, $\A$ be a diagonal matrix in $\mathbb{P}_{n}$, and $\L := \L_{f}(\A)$ be the corresponding Loewner matrix. Then $\L \circ \L \geq 0$ if and only if $\langle\D f(\A)(\X)\,, \D f(\A)(\Y)\rangle_{\F} \geq 0$ for all $\X\,, \Y \in \mathbb{H}^{+}_{n}$\,.

\label{PropositionSeptember20V2}

\end{proposition}

\begin{proof}

First of all, for all $\X\,, \Y \in \mathbb{H}^{+}_{n}$, we get from Lemma \ref{LemmaAdjoint} that
\begin{equation*}
\langle\D f(\A)(\X)\,, \D f(\A)(\Y)\rangle_{\F} = \langle \Phi_{\L}(\X)\,, \Phi_{\L}(\Y)\rangle_{\F} = \langle \X\,, \Phi_{\L \circ \L}(\Y)\rangle_{\F}\,.
\end{equation*}
\begin{itemize}
 \item $\left(\Rightarrow\right):$  Assume that $\L \circ \L \geq 0$. For all $\X\,, \Y \in \mathbb{H}^{+}_{n}$, it follows from Remark \ref{RemarkHadamard} that $\Phi_{\L \circ \L}(\Y) \geq 0$ and $\tr(\X\Phi_{\L \circ \L}(\Y)) \geq 0$. Therefore
\begin{eqnarray*}
0 \leq \tr(\X\Phi_{\L \circ \L}(\Y)) & = & \langle \X\,, \Phi_{\L \circ \L}(\Y)\rangle_{\F} = \langle\D f(\A)(\X)\,, \D f(\A)(\Y)\rangle_{\F} \\
& = & \langle \Phi_{\L}(\X)\,, \Phi_{\L}(\Y)\rangle_{\F}\,.
\end{eqnarray*}
\item $\left(\Leftarrow\right):$ Assume that $\langle\D f(\A)(\X)\,, \D f(\A)(\Y)\rangle_{\F} \geq 0$ for all $\X\,, \Y \in \mathbb{H}^{+}_{n}$. For all $x \in \mathbb{C}^{n}$, we denote by $\T_{x}$ the matrix of $\mathbb{H}^{+}_{n}$ given by $\T_{x} = xx^{*}$. In particular, for every $\Y\in\mathbb H_n^+$ and every $x\in\mathbb C^n$, we have$\tr\left(\T_{x}\Phi_{\L \circ \L}(\Y)\right) \geq 0$. Using that
\begin{equation*}
\tr\left(\T_{x}\Phi_{\L \circ \L}(\Y)\right) = \tr\left(xx^{*}\Phi_{\L \circ \L}(\Y)\right) = \tr(x^{*}\Phi_{\L \circ \L}(\Y)x) = x^{*}\Phi_{\L \circ \L}(\Y)x\,,
\end{equation*}
it follows that $x^{*}\Phi_{\L \circ \L}(\Y)x \geq 0$ for all $x \in \mathbb{C}^{n}$, i.e. $\Phi_{\L \circ \L}(\Y) \geq 0$ for all $\Y \in \mathbb{H}^{+}_{n}$. In particular, for $\Y = \J$ (see Remark \ref{RemarkHadamard}), we get that 
\begin{equation*}
0 \leq \Phi_{\L \circ \L}(\J) = \left(\L \circ \L\right) \circ \J = \L \circ \L\,,
\end{equation*}
i.e. $\L \circ \L \geq 0$\,.
\end{itemize}

\end{proof}

\begin{remark}

Let $\A \in \mathbb{P}_{n}$, and let $\L = \L_{f}(\A)$ be the corresponding Loewner matrix. 
\begin{itemize}
\item Let $\X \in \mathbb{H}_{n}$. As explained in Remark \ref{RemarkHadamard}, if $\X \geq 0$ or $\X \leq 0$, we get $\X \circ \X \geq 0$. However,the converse is not true for arbitrary Hermitian matrices. Indeed, let $\X$ be the matrix in $\mathbb{H}_{2}$ given by
\begin{equation*}
\X = \begin{pmatrix} 1 & 0 \\ 0 & -1 \end{pmatrix}\,.
\end{equation*}
We have $\X \circ \X = \Id_{2}$, so $\X \circ \X > 0$. However, $\X$ is neither positive nor negative\,.
\item Assume that $f$ is an increasing matrix monotone function of degree $n$. For all $\X\,, \Y \in \mathbb{H}^{+}_{n}$, we have 
\begin{equation*}
\D f(\A)(\X)\,, \D f(\A)(\Y) \in \mathbb{H}^{+}_{n}\,.
\end{equation*}
Therefore, 
\begin{equation*}
\tr\left(\D f(\A)(\X)\D f(\A)(\Y)\right) \geq 0\,,
\end{equation*}
i.e. $\langle\D f(\A)(\X)\,, \D f(\A)(\Y)\rangle_{\F} \geq 0$\,.
\end{itemize}

\end{remark}

\noindent The preceding results motivate the following conjecture, which connects radial convexity with matrix monotonicity at each finite order.

\begin{conjecture}

Let $f: \left(0, \infty\right) \to \mathbb{R}$. Then $f$ is radially convex of order $n$ (with respect to the Frobenius norm) if and only if $f$ is matrix monotone of degree $n$\,.

\label{Conjecture}

\end{conjecture}

\noindent By Proposition \ref{FromMonotoneToConvex}, matrix monotonicity of degree $n$ implies radial convexity of order $n$. Thus, it remains to prove the converse implication. The next theorem establishes the conjecture for $n=2$.

\begin{theo}

Let $f:(0,\infty)\to\mathbb{R}$ be differentiable and radially convex of order $2$. Then $f$ is matrix monotone of degree $2$. In other words, Conjecture \ref{Conjecture} holds for $n = 2$\,.

\end{theo}

\begin{proof}

We argue by contraposition. By Proposition \ref{PropositionMonotoneDegreeOne}, radial convexity of order $2$ implies that $f$ is monotone on $(0,\infty)$.  Suppose that $f$ is not matrix monotone of degree $2$. Then there exist $\lambda_{1}\,,\lambda_{2}>0$, with $\lambda_{1} \neq \lambda_{2}$, such that 
\begin{equation*}
ab - \Omega^{2} < 0\,,
\end{equation*}
with $a := f'\left(\lambda_{1}\right)\,, b := f'\left(\lambda_{2}\right)\,,$ and $\Omega := \frac{f\left(\lambda_{1}\right) - f\left(\lambda_{2}\right)}{\lambda_{1}-\lambda_{2}}$.

\noindent Since $f$ is monotone, $a$ and $b$ have the same sign, and therefore $ab \geq 0$. Let
\begin{equation*}
\A =  \begin{pmatrix} \lambda_{1} & 0 \\ 0 & \lambda_{2} \end{pmatrix} \quad \text{ and } \quad \L := \L_{f}\left(\diag(\lambda_{1}, \lambda_{2})\right) = \begin{pmatrix} a & \Omega \\ \Omega & b \end{pmatrix}\,.
\end{equation*}
Recall that for every Hermitian matrix $\H$, we have
\begin{equation*}
\D f(\A)(\H) = \L \circ \H\,.
\end{equation*}
In particular,
\begin{equation*}
\L \circ \L = \begin{pmatrix} a^{2} & \Omega^{2} \\ \Omega^{2} & b^{2} \end{pmatrix}\,,
\end{equation*}
so
\begin{equation*}
\det\left(\L \circ \L\right) =  a^{2}b^{2} - \Omega^{4} = \left(ab-\Omega^{2}\right)\left(ab+\Omega^{2}\right) < 0\,.
\end{equation*}
Hence $\L \circ \L$ is indefinite. Therefore, there exists a vector $v = \begin{pmatrix} v_{1} \\ v_{2} \end{pmatrix} \in \mathbb{R}^{2}$ such that 
\begin{equation*}
v^{t} \left(\L \circ \L\right)v < 0\,.
\end{equation*}
Let $u = \begin{pmatrix} u_{1} \\ u_{2} \end{pmatrix} \in \mathbb{R}^{2}$ and let 
\begin{equation*}
\U :=uu^{t} = \begin{pmatrix} u^{2}_{1}& u_{1}u_{2} \\ u_{1}u_{2} & u^{2}_{2} \end{pmatrix}\,, \qquad \V:=vv^{t} = \begin{pmatrix}
v^{2}_{1} & v_{1}v_{2} \\ v_{1}v_{2} & v^{2}_{2} \end{pmatrix}\,.
\end{equation*}
Clearly, $\U\,, \V \geq 0$ (determinant is $0$, and hermitian because the coefficients are real). Since
\begin{equation*}
\D f(\A)(\U) = \L \circ \U \quad \text{ and } \quad \D f(\A)(\V) = \L \circ \V\,,
\end{equation*}
we obtain
\begin{equation*}
\langle \D f(\A)(\U)\,, \D f(\A)(\V) \rangle_{\F} = \sum\limits_{i,j = 1}^{2} \L^{2}_{ij}u_{i}u_{j}v_{i}v_{j} = \left(u \circ v\right)^{t} \left(\L \circ \L\right)\left(u\circ v\right)\,.
\end{equation*}
Take $u_{1} = u_{2} = 1$, i.e.
\begin{equation*}
\U =  \begin{pmatrix} 1 & 1 \\ 1 & 1 \end{pmatrix} \quad \text{ and } \quad u \circ v = v\,.
\end{equation*}
so it follows that
\begin{equation*}
\langle \D f(\A)(\U)\,, \D f(\A)(\V)\rangle_{\F} = v^{t} \left(\L \circ \L\right)v <0\,.
\end{equation*}
In particular, it follows that $f$ is not radially convex of order $2$ (see Proposition \ref{PropositionSeptember20V1}), contradicting our initial hypothesis on $f$. Therefore $f$ is matrix monotone of degree $2$.

\end{proof}

\begin{remark}

The proof of the conjecture for $n=2$ relies strongly on the particular form of the determinant of a $2 \times 2$ matrix. Indeed, if
\begin{equation*}
\L = \begin{pmatrix} a & b \\ b & c \end{pmatrix}\,,
\end{equation*}
then 
\begin{equation*}
\det(\L \circ \L) = a^{2}c^{2}-b^{4} = \left(ac-b^{2}\right)\left(ac+b^{2}\right) = \det(\L)\left(ac+b^{2}\right)\,.
\end{equation*}
Therefore, once the signs of the diagonal entries are known, the positive semidefiniteness of $\L\circ\L$ provides direct information about the sign of $\det(\L)$. This is the key observation used in the proof of the conjecture in degree $2$.

\noindent The situation becomes considerably more subtle in degree $3$. Let
\begin{equation*}
\L= \begin{pmatrix} a & x & y \\ x & b & z \\ y & z & c \end{pmatrix}\,.
\end{equation*}
A direct computation gives
\begin{equation*}
\det(\L) = abc + 2xyz - az^{2} - by^{2} - cx^{2}\,,
\end{equation*}
whereas
\begin{equation*}
\det(\L\circ\L) = a^{2}b^{2}c^{2} +2x^{2}y^{2}z^{2} -a^{2}z^{4} -b^{2}y^{4} -c^{2}x^{4}\,.
\end{equation*}
In contrast with the $2\times2$ case, there is no analogous factorization involving $\det(\L)$. More importantly, the Hadamard square loses information concerning the signs of the off-diagonal entries. In particular, the term $2xyz$ appearing in $\det(\L)$ is replaced by $2x^{2}y^{2}z^{2}$ in $\det(\L\circ\L)$. Thus, the positivity of $\L \circ \L$ does not, for a general symmetric matrix, determine the sign of $\det(\L)$.

\end{remark}

\begin{remark}

Suppose that $f$ is radially convex of order $n$. Then, by Propositions
\ref{PropositionSeptember20V1} and \ref{PropositionSeptember20V2}, for every
diagonal matrix $\A\in\mathbb P_n$, we have $\L_{f}(\A) \circ \L_{f}(\A) \geq 0$. Let $1 \leq m \leq n$ and let $\B = \diag\left(\lambda_{1}\,, \ldots\,, \lambda_{m}\right) \in \mathbb{P}_{m}$. Choose $\lambda_{m+1},\ldots,\lambda_{n}>0$ and consider $\A = \diag\left(\lambda_{i}\right)_{1 \leq i \leq n} \in \mathbb{P}_{n}$. By construction, $\L_{f}(\B)$ is the leading principal submatrix of size $m$ of
$\L_{f}(\A)$. Therefore, $\L_{f}(\B) \circ \L_{f}(\B)$ is the corresponding principal submatrix of $\L_{f}(\A) \circ \L_{f}(\A)$. Since every principal submatrix of a positive semidefinite matrix is positive semidefinite, we obtain
\begin{equation*}
\L_{f}(\B)\circ\L_{f}(\B)\geq0\,.
\end{equation*}
However, this argument alone does not imply that $f$ is radially convex of order $m$, since the positivity of $\L_{f}(\B)\circ\L_{f}(\B)$ has only been shown to be a necessary condition for radial convexity. This provides a weaker version of
Proposition \ref{PropositionRCNToM}, but shows directly how the corresponding
Loewner matrices behave when passing to a lower degree.

\end{remark}

\noindent The previous remark shows that the argument used for $n=2$ does not extend directly to $n=3$. Nevertheless, radial convexity already provides most of the inequalities required for the positivity of the corresponding Loewner matrix. More precisely, we have the following result.

\begin{proposition}

Let $f: \left(0, \infty\right) \to \mathbb{R}$ be differentiable and increasing, and assume that $f$ is radially convex of order $3$. Then $f$ is matrix monotone of degree $3$ if and only if
\begin{equation*}
\det\left(\L_{f}(\A)\right) \geq 0
\end{equation*}
for every $\A \in \mathbb{P}_{3}$\,.

\end{proposition}

\begin{remark}

Recall that a Hermitian matrix is positive semidefinite if and only if all of its principal minors are nonnegative. More precisely, if
\begin{equation*}
\A = \begin{pmatrix} a & x & y \\ \overline{x} & b & z \\ \overline{y} & \overline{z} & c \end{pmatrix} \in \mathbb{H}_{3}\,,
\end{equation*}
then its principal minors of order $1$ are
\begin{equation*}
a\,,\qquad b\,,\qquad c\,,
\end{equation*}
its principal minors of order $2$ are
\begin{equation*}
ab-\left|x\right|^{2}\,,\qquad ac-\left|y\right|^{2}\,, \qquad bc- \left|z\right|^{2}\,,
\end{equation*}
and its unique principal minor of order $3$ is $\det(\A)$. In particular, $\A \geq 0$ if and only if all seven of these quantities are nonnegative.

\label{RemarkMinorsPSD}

\end{remark}

\begin{proof}

The direct implication follows immediately from Remark \ref{RemarkMatrixMonotone}. Assume that $\det\left(\L_{f}(\A)\right) \geq 0$ for every $\A \in \mathbb{P}_{3}$. The goal is to show that $f$ is matrix monotone of degree $3$.

\noindent Let $\D = \diag\left(\lambda_{1}\,,\lambda_{2}\,,\lambda_{3}\right) \in \mathbb{P}_{3}$ and let 
\begin{equation*}
\L_{f}(\D)= \begin{pmatrix} a_{1} & x & y \\ x & a_{2} & z \\ y & z & a_{3} \end{pmatrix}
\end{equation*}
be the corresponding Loewner matrix (with $a_{i} = f'(\lambda_{i}), 1 \leq i \leq 3$). Since $f$ is increasing, 
\begin{equation*}
f'(\lambda_{i}) \geq 0\,, \qquad \left(1 \leq i \leq 3\right)\,,
\end{equation*}
i.e. all the principal minors of order $1$ of $\L_{f}(\D)$ are nonnegative.

\noindent On the other hand, since $f$ is radially convex of order $3$, it follows from Propositions \ref{PropositionSeptember20V1} and \ref{PropositionSeptember20V2} that 
\begin{equation*}
\L_f(\D)\circ\L_f(\D)
=
\begin{pmatrix}
a_1^2 & x^2 & y^2\\
x^2 & a_2^2 & z^2\\
y^2 & z^2 & a_3^2
\end{pmatrix}\geq0.
\end{equation*}
Therefore, all of its principal minors of order $2$ are nonnegative, and hence
\begin{equation*}
a_1^2a_2^2-x^4\geq0,\qquad
a_1^2a_3^2-y^4\geq0,\qquad
a_2^2a_3^2-z^4\geq0.
\end{equation*}
Since
\begin{equation*}
a^{2}_{1}a^{2}_{2}-x^{4} = \left(a_{1}a_{2}-x^{2}\right)\left(a_{1}a_{2}+x^{2}\right)
\end{equation*}
and $a_{1}, a_{2}\geq 0 $, it follows that $a_{1}a_{2} - x^{2} \geq 0$. Similarly,
\begin{equation*}
a_{1}a_{3} - y^{2} \geq 0 \qquad \text{and} \qquad a_{2}a_{3}-z^{2} \geq 0\,.
\end{equation*}
Thus, all the principal minors of order $2$ of $\L_{f}(\D)$ are nonnegative.

\noindent Finally, by assumption,
\begin{equation*}
\det\left(\L_{f}(\D)\right)\geq0\,.
\end{equation*}
Hence all the principal minors of $\L_{f}(\D)$ are nonnegative. It follows from Remark \ref{RemarkMinorsPSD} that $\L_{f}(\D) \geq 0$. Since $\D\in\mathbb P_3$ was arbitrary, the Loewner characterization implies that $f$ is matrix monotone of degree $3$.

\end{proof}

\begin{remark}

Let $\widetilde{f}_{2}: \left(0, \infty\right) \to \left(0, \infty\right)$ be the function introduced in Example \ref{ExampleMonotone}, i.e.
\begin{equation*}
\widetilde{f}_{2}(x) = \frac{x}{\sqrt{2}(1+x)} + \frac{1}{3}\left(\frac{x}{\sqrt{2}(1+x)}\right)^{3}\,.
\end{equation*}
The function $\widetilde{f}_{2}$ is matrix monotone (and increasing) of degree $2$ but not $3$. It follows from Proposition \ref{FromMonotoneToConvex} that $\widetilde{f}_{2}$ is radially convex of order $2$. However, one can see that the function $\widetilde{f}_{2}$ is not radially convex of order 3. Indeed, let $\D_{3} = \diag\left(1\,, 2\,, 3\right)$. The Loewner matrix associated with $\widetilde{f}_{2}$ and $\D_{3}$ is given by
\begin{equation*}
\L_{\widetilde{f}_{2}}(\D_{3}) = \begin{pmatrix} \frac{9\sqrt{2}}{64} & \frac{253\sqrt{2}}{2592} & \frac{115\sqrt{2}}{1536} \\ \frac{253\sqrt{2}}{2592} & \frac{11\sqrt{2}}{162} & \frac{1081\sqrt{2}}{20736} \\ \frac{115\sqrt{2}}{1536} & \frac{1081\sqrt{2}}{20736} & \frac{41\sqrt{2}}{1024} \end{pmatrix}\,.
\end{equation*}
Then
\begin{equation*}
\det\left(\L_{\widetilde{f}_{2}}(\D_{3})\circ\L_{\widetilde{f}_{2}}(\D_{3})\right) = -\frac{118534087}{851946664988327804928} < 0\,.
\end{equation*}
Therefore, the matrix $\L_{\widetilde{f}_{2}}(\D_{3})\circ\L_{\widetilde{f}_{2}}(\D_{3})$ is not positive semidefinite, so $\widetilde{f}_{2}$ is not radially convex of order $3$.

\noindent Similarly, let $\widetilde{f}_{3}: \left(0, \infty\right) \to \left(0, \infty\right)$ be the function introduced in Example \ref{ExampleMonotone}, i.e.
\begin{equation*}
\widetilde{f}_{3}(x) = \frac{x}{5(1+x)} + \frac{1}{3}\left(\frac{x}{5(1+x)}\right)^{3} + \frac{1}{5}\left(\frac{x}{5(1+x)}\right)^{5}\,.
\end{equation*}
Recall that $\widetilde{f}_{3}$ is matrix monotone of degree $3$, but not of degree $4$, on $(0,\infty)$. Consider $\D_{4}=\diag\left(1\,,2\,,3\,,4\right)$. A direct computation gives
\begin{equation*}
\det\left(\L_{\widetilde{f}_{3}}(\D_{4}) \circ \L_{\widetilde{f}_{3}}(\D_{4})\right) < 0\,.
\end{equation*}
Therefore, the matrix $\L_{\widetilde{f}_{3}}(\D_{4}) \circ \L_{\widetilde{f}_{3}}(\D_{4})$ is not positive semidefinite, and hence $\widetilde{f}_{3}$ is not radially convex of order $4$.

\end{remark}

\section{Concluding remarks and future directions}

\label{SectionTheEnd}

The main purpose of this paper was to investigate the relationship between matrix monotonicity and radial convexity on the cone of positive definite matrices. In Section~\ref{SectionMtoRC}, we showed that matrix monotonicity implies radial convexity for a large family of metrics. In particular, this is the case for metrics of the form
\begin{equation*}
\d_{\left\|\cdot\right\|,f}(\A,\B) = \left\|f(\A)-f(\B)\right\|\,,
\end{equation*}
where $\left\|\cdot\right\|$ is a unitarily invariant norm and $f$ is an appropriate matrix monotone function. We also established radial convexity for several other natural metrics on the cone of positive definite matrices.

\bigskip

\noindent In Section~\ref{SectionRCtoM}, we considered the converse problem for the metric induced by the Frobenius norm. This led us to conjecture that, for every $n \geq 1$, radial convexity of order $n$ is equivalent to matrix monotonicity of degree $n$. We proved this conjecture  for $n=2$. More generally, we showed that radial convexity imposes a positivity condition on the Hadamard square of the corresponding Loewner matrix, providing a direct link between the metric geometry of the positive definite cone and the classical Loewner characterization of matrix monotonicity.

\bigskip

\noindent These results suggest several directions for future work. A first natural problem is to investigate the converse implication for arbitrary unitarily invariant norms. More precisely, given a unitarily invariant norm $\left\|\cdot\right\|$, one may ask whether radial convexity with respect to $\d_{\|\cdot\|,f}$ characterizes matrix monotonicity of $f$ at each finite degree. In the case of the Frobenius norm, the differential characterization obtained in Section \ref{SectionRCtoM} leads to a condition involving the Hadamard square of the Loewner matrix. It would be interesting to determine the corresponding infinitesimal condition for other unitarily invariant norms and to obtain an analogue of the implications established in that section.

\bigskip

\noindent A second direction is to move beyond metrics induced by unitarily invariant norms and investigate similar phenomena for other natural geometries on the positive definite cone. The broader objective is to develop alternative geometric and metric characterizations of matrix monotonicity. In particular, finite-degree matrix monotonicity exhibits a subtle dependence on the matrix size: a function may be matrix monotone of degree $n$ but fail to be matrix monotone of degree $n+1$. Understanding radial convexity for different metrics may provide new ways to identify the mechanisms responsible for this loss of monotonicity and, more generally, to better understand the distinction between the different degrees of matrix monotonicity.

\bigskip

\noindent Returning to the first direction, we conclude by showing that the equivalence obtained in degree $2$ for the Frobenius norm extends, under a mild non-degeneracy assumption, to every unitarily invariant norm. In Theorem \ref{TheoremUnitaryInvariant}, we show that this additional assumption can in fact be removed. 

\begin{proposition}

Let $f: \left(0, \infty\right) \to \left(0, \infty\right)$ be differentiable and increasing, and let $\left\|\cdot\right\|$ be a unitarily invariant norm. Assume that $f$ is not matrix monotone of degree $2$. Suppose moreover that there exist $a \neq b$ such that $\L_{f}(a,b)$ is not positive semidefinite and
\begin{equation*}
f'(a)f'(b)>0\,.
\end{equation*}
Then there exist $\A, \B, \C \in \mathbb{P}_{2}$ such that
\begin{equation*}
\A \leq \C \leq \B \qquad \text{ and } \qquad \left\|f(\A) - f(\C)\right\| > \left\|f(\A) - f(\B)\right\|\,.
\end{equation*}

\end{proposition}

\begin{remark}

For any symmetric matrix $\L$ and any diagonal matrix $\Delta$ with real entries, we have
\begin{equation*}
\L \circ \left(\Delta\H\Delta\right) = \left(\Delta\L\Delta\right) \circ \H\,.
\end{equation*}
    
\end{remark}

\begin{proof}

Let $\A$ be the matrix in $\mathbb{P}_{2}$ given by
\begin{equation*}
\A = \begin{pmatrix} a & 0 \\ 0 & b \end{pmatrix}\,, \qquad \left(a \neq b\right)\,.
\end{equation*}
with $f'(a)f'(b) > 0$ and such that the corresponding matrix $\L:=\L_{f}(\A)$ is not positive semidefinite. Since $f$ is increasing,
\begin{equation*}
f'(a)\,,f'(b)\,,f^{[1]}(a,b) \geq 0\,.
\end{equation*}
Using that $\L$ is not positive semidefinite,
\begin{equation*}
\det(\L) = f'(a)f'(b)-f^{[1]}(a,b)^{2} < 0\,.
\end{equation*}
Let $\Delta$ be the diagonal matrix given by
\begin{equation*}
\Delta= \diag\left(f'(a)^{-1/2}\,, f'(b)^{-1/2}\right)\,.
\end{equation*}
Then
\begin{equation*}
\R(a, b) := \Delta\L\Delta = \begin{pmatrix} 1 & \frac{f^{[1]}(a,b)}{\sqrt{f'(a)f'(b)}} \\ \frac{f^{[1]}(a,b)}{\sqrt{f'(a)f'(b)}} &1 \end{pmatrix}\,.
\end{equation*}
To simplify the notations, let
\begin{equation*}
k = \sqrt{f'(a)f'(b)} \qquad \text{and} \qquad \rho=\frac{f^{[1]}(a,b)}{k}\,, 
\end{equation*}
i.e. 
\begin{equation*}
\R(\rho) := \R(a,b) = \begin{pmatrix} 1 & \rho \\ \rho & 1\end{pmatrix}\,.
\end{equation*}
Note that $\rho > 1$.

\noindent We now consider the vectors
\begin{equation*}
u =  \begin{pmatrix} 1 \\ 1 \end{pmatrix}\,, \qquad v= \begin{pmatrix} 1 \\ -1 \end{pmatrix}\,.
\end{equation*}
and let $\U$ and $\V$ be the rank-one matrices given by
\begin{equation*}
\U := uu^{t} = \begin{pmatrix} 1 & 1 \\ 1 & 1 \end{pmatrix}\,, \qquad \V := vv^{t} = \begin{pmatrix} 1 & -1 \\ -1 & 1 \end{pmatrix}
\end{equation*}
We get
\begin{equation*}
\R(\rho) \circ \U = \begin{pmatrix} 1 & \rho \\ \rho & 1 \end{pmatrix} = \R(\rho) \qquad \R(\rho) \circ \V = \begin{pmatrix} 1 & -\rho \\ -\rho & 1 \end{pmatrix}\,.
\end{equation*}
Let $\P$ and $\Q$ be the matrices given by
\begin{equation*}
\P = \alpha\Delta\U\Delta \qquad \text{and} \qquad \Q = \beta\Delta\V\Delta\,.
\end{equation*}
Since $\U$ and $\V$ are positive semidefinite of rank one, so are $\P$ and $\Q$.

\noindent Let $t>0$ and define
\begin{equation*}
\C_{t} = \A + t\P \qquad \text{ and } \qquad \B_{t} = \A + t\left(\P + \Q\right)\,.
\end{equation*}
Since $\P$ and $\Q$ are positive semidefinite, we have
\begin{equation*}
\A \leq \C_{t} \leq\B_{t}\,.
\end{equation*}

\noindent Using that 
\begin{equation*}
\L \circ \left(\Delta\H\Delta\right) = \left(\Delta\L\Delta\right) \circ \H \qquad \text{ and } \qquad \Delta\L\Delta = \R(\rho)\,,
\end{equation*}
then
\begin{equation*}
\D f(\A)(\P) = \L \circ \left(\alpha\Delta\U\Delta\right) = \alpha\R(\rho) \circ \U = \alpha \begin{pmatrix} 1 & \rho \\ \rho & 1 \end{pmatrix} = \alpha \R(\rho)\,.
\end{equation*}
Similarly,
\begin{equation*}
\D f(\A)(\P+\Q) = \R(\rho) \circ \left(\alpha\U + \beta\V\right) = \begin{pmatrix} \alpha + \beta & \rho(\alpha-\beta) \\ \rho(\alpha-\beta) & \alpha+\beta \end{pmatrix}\,.
\end{equation*}
Since $\rho>1$, the eigenvalues of $\D f(\A)(\P)$ are
\begin{equation*}
\alpha\left(1+\rho\right) \qquad \text{ and } \qquad \alpha\left(1-\rho\right)\,,
\end{equation*}
and hence its singular values are
\begin{equation*}
\alpha(\rho+1)
\qquad\text{and}\qquad
\alpha(\rho-1)\,.
\end{equation*}
On the other hand, the eigenvalues of $\D f(\A)(\P+\Q)$ are
\begin{equation*}
\lambda_{1} = \alpha\left(\rho+1\right) - \beta\left(\rho-1\right) \qquad \text{ and } \qquad \lambda_{2} = -\alpha\left(\rho-1\right) + \beta\left(\rho+1\right)\,.
\end{equation*}
We would like the absolute values of both eigenvalues to be strictly smaller than the corresponding singular values of $\D f(\A)(\P)$. The first eigenvalue decreases as soon as $\beta>0$. For the second one, it is enough to require that it stays negative, so
\begin{equation*}
-\alpha\left(\rho-1\right) + \beta\left(\rho+1\right) < 0\,.
\end{equation*}
We therefore choose $\alpha>0$ and $\beta>0$ such that
\begin{equation*}
0 < \beta < \alpha \frac{\rho-1}{\rho+1}\,.
\end{equation*}
With this choice, $\lambda_{1} > 0$ and $\lambda_{2} < 0$, and the singular values of $\D f(\A)(\P+\Q)$ are
\begin{equation*}
\alpha(\rho+1)-\beta(\rho-1) \qquad \text{ and } \qquad \alpha(\rho-1)-\beta(\rho+1)\,.
\end{equation*}
Therefore
\begin{equation*}
s_{j}\left(\D f(\A)(\P+\Q)\right) < s_{j}\left(\D f(\A)(\P)\right)\,, \qquad \left(j = 1, 2\right)\,.
\end{equation*}
Since $\left\|\cdot\right\|$ is unitarily invariant, it follows that
\begin{equation*}
\left\|\D f(\A)(\P+\Q)\right\| < \left\|\D f(\A)(\P)\right\|\,.
\end{equation*}
\noindent Using that
\begin{equation*}
f(\C_{t}) - f(\A) = t\D f(\A)(\P)+ o(t)
\end{equation*}
and 
\begin{equation*}
f(\B_{t}) - f(\A) = t\D f(\A)(\P+\Q) + o(t)\,.
\end{equation*}
we get
\begin{equation*}
\left\|f(\C_{t})-f(\A)\right\| = t\left\|\D f(\A)(\P)\right\| + o(t)
\end{equation*}
and 
\begin{equation*}
\left\|f(\B_{t}) - f(\A)\right\| = t\left\|\D f(\A)(\P+\Q)\right\| + o(t)\,.
\end{equation*}
Finally, since
\begin{equation*}
\left\|\D f(\A)(\P+\Q)\right\| < \left\|\D f(\A)(\P)\right\|\,,
\end{equation*}
for all $t>0$ small enough, we obtain
\begin{equation*}
\left\|f(\A)-f(\B_{t})\right\| < \left\|f(\A)-f(\C_{t})\right\|\,,
\end{equation*}
this proves the result.

\end{proof}

\begin{remark}

An analogous result holds when $f$ is differentiable and decreasing; the proof is similar.

\end{remark}

\begin{theo}

Let $f: \left(0, \infty\right) \to \mathbb{R}$ be differentiable, and let $\|\cdot\|$ be a unitarily invariant norm. Then $f$ is radially convex of order $2$ with respect to the metric $\d_{\|\cdot\|,f}$ if and only if $f$ is matrix monotone of degree $2$.
\label{TheoremUnitaryInvariant}

\end{theo}

\begin{proof}

If $f$ is matrix monotone of degree $2$, then Proposition \ref{FromMonotoneToConvex} shows that $f$ is radially convex of order $2$ with respect to $\d_{\|\cdot\|,f}$.

\noindent For the converse, we argue by contraposition. In view of the preceding proposition, it is enough to show that the additional assumption
\begin{equation*}
f'(a)f'(b) > 0
\end{equation*}
can always be satisfied. Suppose that $f$ is not matrix monotone of degree $2$. If $f$ is not monotone on $\left(0, \infty\right)$, then Proposition \ref{PropositionMonotoneDegreeOne} implies that $f$ is not radially convex of order $2$, and there is nothing to prove. We may therefore assume that $f$ is monotone. Without loss of generality, we may assume that $f$ is increasing; otherwise, replace $f$ by $-f$, which induces the same metric. Then there exist $a < b$ such that
\begin{equation*}
\L_{f}(a,b) = \begin{pmatrix} f'(a) & f^{[1]}(a,b) \\ f^{[1]}(a,b) & f'(b) \end{pmatrix}
\end{equation*}
is not positive semidefinite. Since $f$ is increasing,
\begin{equation*}
f'(a) \geq 0\,, \qquad f'(b) \geq 0\,, \qquad f^{[1]}(a,b) \geq 0\,.
\end{equation*}
Hence
\begin{equation*}
f'(a)f'(b) - f^{[1]}(a,b)^{2} < 0\,.
\end{equation*}
In particular, $f^{[1]}(a,b) > 0$.

\noindent If $f'(a) > 0$ and $f'(b) > 0$, then the hypotheses of the preceding proposition are satisfied, and the result follows. It remains to consider the case in which at least one of $f'(a)$ and $f'(b)$ is zero. Suppose first that $f'(a)=0$ and $f'(b)>0$. By the mean value theorem, there exists $c \in \left(a, b\right)$ such that
\begin{equation*}
f'(c) = f^{[1]}(a,b) > 0\,.
\end{equation*}
By Darboux's theorem for derivatives, for every sufficiently small $\varepsilon > 0$ there exists $a_{\varepsilon} \in \left(a, c\right)$ such that
\begin{equation*}
f'(a_{\varepsilon}) = \varepsilon.
\end{equation*}
Since $f'(c)>0$ and $f$ is increasing, we have $f(b) > f(c)$. Hence
\begin{equation*}
f^{[1]}(a_{\varepsilon}, b) = \frac{f(b)-f(a_{\varepsilon})}{b-a_{\varepsilon}} \geq \frac{f(b)-f(c)}{b-a} > 0\,.
\end{equation*}
Therefore, for $\varepsilon>0$ sufficiently small,
\begin{equation*}
f'(a_{\varepsilon})f'(b)-f^{[1]}(a_{\varepsilon},b)^{2} = \epsilon f'(b)-f^{[1]}(a_{\varepsilon},b)^{2} < 0\,.
\end{equation*}
Thus $\L_{f}(a_{\varepsilon},b)$ is not positive semidefinite and
\begin{equation*}
f'(a_{\varepsilon})f'(b)>0\,.
\end{equation*}
The case $f'(a)>0$ and $f'(b)=0$ is analogous. Finally, suppose that
\begin{equation*}
f'(a) = f'(b) = 0\,.
\end{equation*}
Again, by the mean value theorem there exists $c \in \left(a, b\right)$ such that
\begin{equation*}
f'(c) = f^{[1]}(a,b) > 0\,.
\end{equation*}
Choose points $c_{1}, c_{2}$ such that
\begin{equation*}
a < c_{1} < c < c_{2} < b
\end{equation*}
and
\begin{equation*}
f'(c_{1}) > 0\,, \qquad f'(c_{2}) > 0\,.
\end{equation*}
By Darboux's theorem, for every sufficiently small $\varepsilon > 0$ there exist
\begin{equation*}
a_{\varepsilon} \in \left(a, c_{1}\right)\,, \qquad b_{\varepsilon} \in \left(c_{2}, b\right)
\end{equation*}
such that
\begin{equation*}
f'(a_{\varepsilon}) = f'(b_{\varepsilon}) = \varepsilon\,.
\end{equation*}
Since $f$ is increasing and $c_{1} < c_{2}$,
\begin{equation*}
f^{[1]}(a_{\varepsilon}, b_{\varepsilon}) \geq \frac{f(c_{2})-f(c_{1})}{b-a}>0\,.
\end{equation*}
Therefore, for $\varepsilon>0$ sufficiently small,
\begin{equation*}
f'(a_{\varepsilon})f'(b_{\varepsilon}) - f^{[1]}(a_{\varepsilon},b_{\varepsilon})^{2} = \varepsilon^{2} - f^{[1]}(a_{\varepsilon},b_{\varepsilon})^{2} < 0\,.
\end{equation*}
Hence $\L_{f}(a_{\varepsilon}, b_{\varepsilon})$ is not positive semidefinite and
\begin{equation*}
f'(a_{\varepsilon})f'(b_{\varepsilon}) > 0\,.
\end{equation*}
Thus, whenever $f$ fails to be matrix monotone of degree $2$, we can find a pair satisfying the hypotheses of the preceding proposition. It follows that $f$ is not radially convex of order $2$ with respect to $\d_{\|\cdot\|,f}$. Therefore, radial convexity of order $2$ implies matrix monotonicity of degree $2$\,.

\end{proof}

\end{document}